\documentclass{amsart}
\usepackage{amsfonts}
\usepackage{amsmath}
\usepackage{amssymb}
\usepackage{amsthm}
\usepackage{color}
\usepackage{graphicx}
\usepackage{xcolor}
\usepackage{xspace}
\usepackage{enumerate}

\newtheorem{theorem}{Theorem}[section]
\newtheorem{lemma}[theorem]{Lemma}
\newtheorem{proposition}[theorem]{Proposition}
\newtheorem{corollary}[theorem]{Corollary}

\theoremstyle{definition}

\newtheorem{example}{Example}

\newcommand{\ie}{i.{}e.{}\xspace} 
\newcommand{\eg}{e.{}g.{}\xspace} 
\newcommand{\bbN}{\mathbb{N}} 
\newcommand{\bbR}{\mathbb{R}} 
\newcommand{\cG}{\mathcal{G}} 
\newcommand{\cR}{\mathcal{R}} 
\newcommand{\cS}{\mathcal{S}} 
\newcommand{\cT}{\mathcal{T}} 
\newcommand{\bd}{\mathbf{d}} 
\newcommand{\bm}{\mathbf{m}} 
\newcommand{\bn}{\mathbf{n}} 
\newcommand{\set}[2]{\left\{#1\,\middle|\,#2\right\}} 
\DeclareMathOperator{\supp}{supp} 
\newcommand{\eqdef}{\mbox{\,\raisebox{0.2ex}{\scriptsize\ensuremath{\mathrm:}}\ensuremath{=}\,}} 
\newcommand{\defn}[1]{{\sl{\color{blue}#1}}} 
\newcommand\svx{source\xspace} 
\newcommand{\dvx}{sink\xspace} 

\begin{document}

\title[Diameter of a subfamily of transportation polytopes]{A note on the diameter of transportation polytopes with prescribed source degrees}

\author{Henning Bruhn-Fujimoto}
\address{\'Equipe Combinatoire et Optimisation, Universit\'e Pierre et Marie Curie, Paris}
\email{bruhn@math.jussieu.fr}
\urladdr{http://www.math.jussieu.fr/~bruhn/}

\author{Guillaume Chapuy}
\address{CNRS \& LIAFA, Universit\'e Paris-Diderot, Paris}
\email{guillaume.chapuy@liafa.jussieu.fr}
\urladdr{http://www.liafa.jussieu.fr/~chapuy/}

\author{Edward D. Kim}
\address{Department of Mathematics, POSTECH, Pohang}
\email{edwardkim@postech.ac.kr}
\urladdr{http://math.postech.ac.kr/~edwardkim/}

\author{Vincent Pilaud}
\address{CNRS \& LIX, \'Ecole Polytechnique, Palaiseau}
\email{vincent.pilaud@lix.polytechnique.fr}
\urladdr{http://www.lix.polytechnique.fr/~pilaud/}

\thanks{
EK was supported by Vidi grant 639.032.917 of the Netherlands Organization for Scientific Research (NWO), and by a BK21 grant of the Korea Research Foundation. VP was partially supported by grant MTM2008-04699-C03-02 and MTM2011-22792 of the Spanish Ministerio de Ciencia e Innovaci\'on, and by a postdoctoral grant of the Fields Institute of Toronto.
}

\begin{abstract}
Brightwell, van den Heuvel and Stougie proved that the diameter of an~$m \times n$ transportation polytope is at most~$8(m+n-2)$, a factor of eight away from the Hirsch Conjecture. This bound was improved to~$3(m+n-1)$ by Hurkens. We investigate diameters for certain classes of transportation polytopes.

\bigskip
\noindent\textsc{Note.}
After the completion of this note, we discovered that the class of transportation polytopes studied in this note was already considered in

\medskip
Michel L.~Balinski. On two special classes of transportation polytopes. \emph{Math. Programming Stud.}, 1:43--58, 1974.

Michel L.~Balinski and Fred J.~Rispoli. Signature classes of transportation polytopes. \emph{Mathematical Programming}, 60(2, Ser.~A):127--144, 1993.

\medskip
These papers contain both refinements of our results and generalizations to more general classes of transportation problems. In view of these papers, this note will not be submitted for publication. \end{abstract}

\maketitle


\section{Introduction}

In~\cite{Santos}, Santos provided the first counterexample to the famous \defn{Hirsch Conjecture}, which asserted that the diameter of the~$1$-skeleton of a~$\delta$-dimensional polytope with~$\phi$ facets is at most~$\phi - \delta$. This bound, which we call the \defn{Hirsch bound}, was however natural and plausible. Although there is little hope to characterize the polytopes which satisfy this bound, it is interesting to investigate families of polytopes for which the bound may hold. The  transportation polytopes constitute one candidate for such a family.

Consider a \defn{supply function}~$\bm : M \to \bbR_{\ge 0}$ on a set~$M$ of~$m$ \defn{\svx{}s} and a \defn{demand function}~$\bn : N \to \bbR_{\ge 0}$ on a set~$N$ of~$n$ \defn{\dvx{}s}. The $m \times n$ \defn{transportation polytope} $P_{\bm,\bn}$ is a convex polytope obtained by intersecting the positive orthant~$\bbR^{M \times N}_{\ge 0}$ with the following~$m+n$ affine hyperplanes: 
\[ P_{\bm,\bn} \eqdef 
\set{(x_{\mu,\nu}) \in \bbR^{M \times N}_{\ge 0}}{
\forall \mu,\; \sum\nolimits_{\nu} x_{\mu,\nu} = \bm(\mu) 
\text{ and }
\forall \nu,\; \sum\nolimits_{\mu} x_{\mu,\nu} = \bn(\nu) 
}.\]
We always assume that~${\sum_\mu \bm(\mu) = \sum_\nu \bn(\nu)}$, such that $P_{\bm,\bn}$ is non-empty. Intuitively, a point of the transportation polytope~$P_{\bm,\bn}$ is an assignment of quantities to be transported between the \svx{}s of~$M$ and the \dvx{}s of~$N$ on each edge of the complete bipartite graph~$K_{M,N}$ such that the total quantity that a \svx~$\mu$ provides corresponds to its supply~$\bm(\mu)$ while the total quantity that a \dvx~$\nu$ receives corresponds to its demand~$\bn(\nu)$. Optimizing transportation costs naturally gives rise to linear optimization problems on transportation polytopes, and thus leads to the question to evaluate the diameter of transportation polytopes.

Any non-empty $m \times n$ transportation polytope has dimension $(m-1)(n-1)$ and at most $mn$ many facets. Consequently, if the Hirsch Conjecture is true for transportation polytopes, the diameter of any $m \times n$ transportation polytope will be at most~${\phi - \delta \le m+n-1}$. In~\cite{BrightwellvdHeuvelStougie} Brightwell, van den Heuvel, and Stougie gave a bound of~$8(m+n-2)$, which was then improved to~$3(m+n-1)$ by Hurkens~\cite{Hurkens}. However, it is not clear so far whether the transportation polytopes satisfy the Hirsch bound or not.

In this note, we study the diameter of a specific subfamily of transportation polytopes. Given a \defn{degree function}~$\bd : M \to \bbN_{\ge 0}$, we consider the $m \times n$ transportation polytope~$P_\bd \eqdef P_{\bm,\bn}$ where the supply and demand functions are defined by~$\bm(\mu) \eqdef 1 + m\bd(\mu)$ and $\bn(\nu) \eqdef m$. (It forces $n = 1 + \sum_\mu \bd(\mu)$, to ensure that total supply equals total demand.) We show that the diameter of~$P_\bd$ does not exceed twice the Hirsch bound. Furthermore, when $\bd\equiv 1$, the polytope~$P_\bd$ is an $m \times (1+m)$ generalized Birkhoff polytope, and we prove that its diameter is precisely given by the Hirsch bound. Along with the diameters of Birkhoff polytopes~\cite{YemelichevKovalevKravtsov} and a recent extension to partition polytopes by Borgwardt~\cite{Borgwardt}, this provides one of the first subfamilies of transportation polytopes with this property.

Our study of the polytope~$P_\bd$ essentially relies on its interesting combinatorial structure: the $1$-skeleton of~$P_\bd$ is (isomorphic to) the pivoting graph on spanning trees of the complete bipartite graph~$K_{M,N}$ where each \svx~$\mu$ has prescribed degree~$1+\bd(\mu)$. We first study independently this family of spanning trees, and relate it later on with the transportation polytope~$P_\bd$.


\section{Bipartite spanning trees with specified \svx degrees}\label{section:trees}

Denote by~$K_{M,N}$ the complete bipartite graph between a set~$M$ of~$m$ \defn{\svx{}s} and a set~$N$ of~$n$ \defn{\dvx{}s}. Fix a function~$\bd : M \to \bbN_{\ge 0}$ and assume that~${n = 1 + \sum_{\mu} \bd(\mu)}$. We consider the set~$\cT_\bd$ of spanning trees of~$K_{M,N}$ where every \svx~$\mu$ has degree~${1 + \bd(\mu)}$. Note that we prescribe the \svx degrees of these trees, but not their \dvx degrees. Any such tree has~$e \eqdef \sum_{\mu} 1 + \bd(\mu)$ edges. The value of~$n$ was chosen so that~$e = m+n-1$, and thus~$\cT_\bd$ is nonempty. We first give a closed formula for the cardinality of~$\cT_\bd$.

\begin{proposition}\label{proposition:count}
The number of trees in~$\cT_\bd$ is precisely
\[ 
|\cT_\bd| = \frac{(n-1)! n^{m-1}}{\prod_\mu \bd(\mu)!}.
\]
\end{proposition}

\begin{proof}
We adapt Andr\'e Joyal's proof of Cayley's formula. We fix a total order~$\prec$ on~$M$ and a special \dvx~$\nu_* \in N$.
Let~$\cR_\bd$ be the set of trees of~$\cT_\bd$ oriented towards \dvx~$\nu_*$ and with a marked \dvx (which may or may not coincide with~$\nu_*$). The number of trees in~$\cR_\bd$ is~$n |\cT_\bd|$. Let~$\cS_\bd$ be the set of directed subgraphs of~$K_{M,N}$ where every \svx~$\mu$ has in-degree~$\bd(\mu)$, and every \svx and \dvx has out-degree~$1$, except~$\nu_*$, which has out-degree~$0$. Since~$\sum_\mu \bd(\mu) = n-1$, the number of graphs in~$\cS_\bd$ is~$\frac{(n-1)! n^{m}}{\prod_\mu \bd(\mu)!}.$ We now define inverse bijections~${\Phi : \cS_\bd \to \cR_\bd}$ and~${\Psi : \cR_\bd \to \cS_\bd}$.

Consider a directed graph~$S$ in~$\cS_\bd$. The connected component~$S_1$ of~$S$ containing \dvx~$\nu_*$ is a tree oriented towards~$\nu_*$. The other connected components~$S_2, \dots, S_p$ of~$S$ are graphs of functions: each is formed by a directed (\svx-\dvx alternating) cycle together with some rooted trees which are oriented towards their roots and glued to the cycle by their roots. For~$k \ge 2$, we denote by~$\mu_k$ the \svx of the cycle of~$S_k$ which is maximal for~$\prec$, and by~$\nu_k$ the \dvx following~$\mu_k$ in this cycle (fix also~$\nu_1 = \nu_*$). By reordering, we can assume without loss of generality that~$\mu_2 \prec \dots \prec \mu_p$. See Figure~\ref{figure:S-decomposition}. To obtain the oriented tree~$\Phi(S)$ from~$S$, we open the cycle in each~$S_k$ by erasing the arc between~$\mu_k$ and~$\nu_k$, we concatenate the resulting trees by adding an arc between~$\mu_k$ and~$\nu_{k-1}$, and we mark \dvx~$\nu_p$. In other words,
\[ \Phi(S) \eqdef (S \setminus \set{(\mu_k,\nu_k)}{k \in \{2,\dots,p\}} ) \cup \set{(\mu_k,\nu_{k-1})}{k \in \{2,\dots,p\}}.
\]
\begin{figure}[hbt]
	\centerline{\includegraphics[scale=.7]{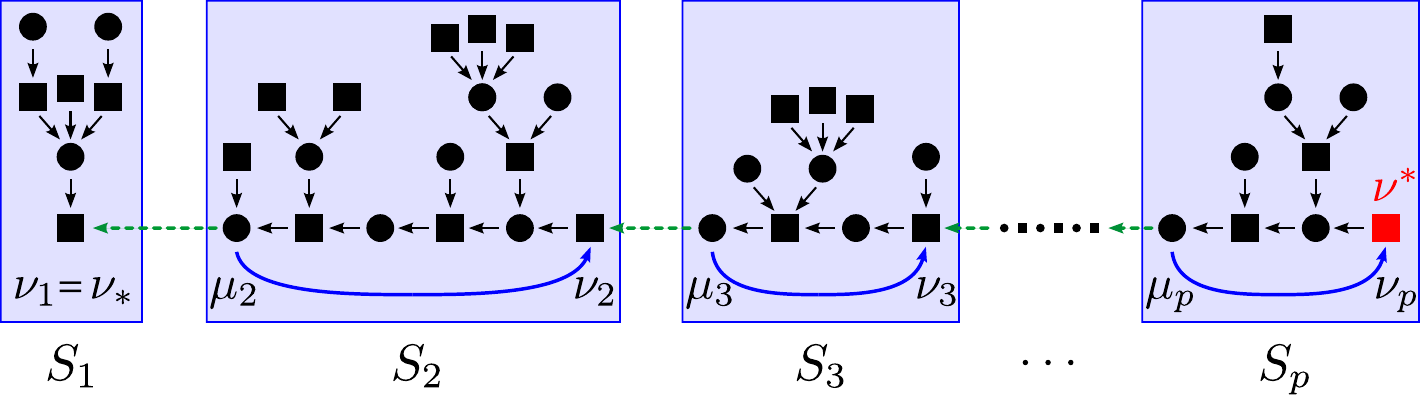}}
	\caption{Schematic decomposition of $S \in \mathcal{S}_\bd$ into $S_1,\dots,S_p$. The map $\Phi$ removes the long blue arcs and inserts the dashed green arcs.}
	\label{figure:S-decomposition}
\end{figure}

Consider an oriented tree~$R$ in~$\cR_\bd$ with marked \dvx~$\nu^*$ (which may or may not be~$\nu_*$). Consider the path from~$\nu_*$ to~$\nu^*$ (note that this path goes against the orientation on~$R$). Denote by $\mu_2$ the \svx\ succeeding $\nu_*$, and inductively choose $\mu_3,\dots,\mu_p$ so that $\mu_i$ is for all~$i$ the first \svx\ on the path that is greater than~$\mu_{i-1}$ with respect to~$\prec$. For~$k \le p-1$, let~$\nu_k$ be the \dvx preceding~$\mu_{k+1}$ along this path, and define~$\nu_p := \nu^*$. See Figure~\ref{figure:R-decomposition}. To obtain the directed graph~$\Psi(R)$ from~$R$, we disconnect~$R$ by erasing the arc between~$\mu_k$ and~$\nu_{k-1}$ and we create cycles by adding an arc between~$\mu_k$ and~$\nu_k$. In other words,
\[ \Psi(R) \eqdef (R \setminus \set{(\mu_k,\nu_{k-1})}{k \in \{2,\dots,p\}} ) \cup \set{(\mu_k,\nu_k)}{k \in \{2,\dots,p\}}.
\]
\begin{figure}[hbt]
	\centerline{\includegraphics[scale=.7]{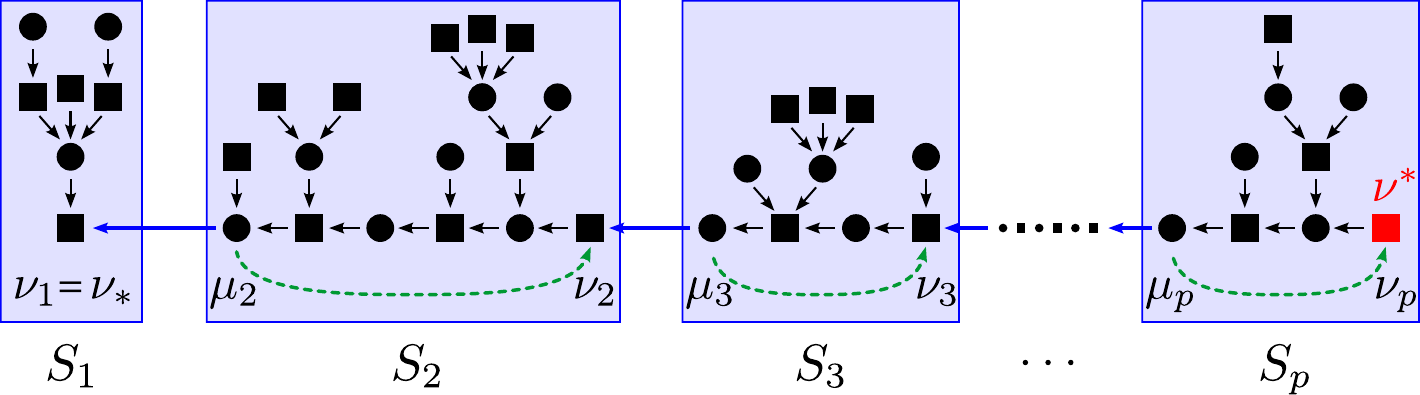}}
	\caption{Schematic depiction of $R \in \mathcal{R}_\bd$ with marked sink $\nu^*$. The map $\Phi$ removes the long blue arcs and inserts the dashed green arcs.}
	\label{figure:R-decomposition}
\end{figure}

It is clear that~$\Phi$ and~$\Psi$ are inverse bijections. We thus obtain that
\[
|\cT_\bd| = \frac{|\cR_\bd|}{n} = \frac{|\cS_\bd|}{n} = \frac{(n-1)!n^{m-1}}{\prod_\mu \bd(\mu)!}.\qedhere
\]
\end{proof}

When~$\bd \equiv d \in \bbN_{\ge 0}$ is constant on~$M$, the number of trees in~$\cT_\bd$ is~$\frac{(n-1)!n^{m-1}}{(d!)^m}$, which recovers a formula of Klee and Witzgall (see~\cite{KleeWitzgall} or~Theorem~2.4 in Chapter~6 of~\cite{YemelichevKovalevKravtsov}). This counting can alternatively be derived from classical generatingfunctionology. For that, consider the collection of all bipartite trees whose \svx{}s are unlabeled and have degree~$d$, whose \dvx{}s are labeled, and which are rooted at an arbitrary \dvx. Denote by~$\theta_d(n)$ the number of such trees with~$n$ \dvx{}s, and let~$\Theta_d(z) \eqdef \sum_n \theta_d(n)z^n/n!$ be the corresponding exponential generating function. According to the dictionary of generating functions~\cite{FlajoletSedgewick}, this function satisfies $\Theta_d(z) = z\zeta_d(\Theta_d(z))$, where $\zeta_d(y) = \exp\left(y^d/d!\right)$. The Lagrange inversion formula ensures that~$\theta_d(n) = n!\,[z^n]\Theta_d(z) = (n-1)!\,[y^{n-1}]\zeta_d(y)^n$, where $[x^p]f(x)$ denotes the coefficient of~$x^p$ in the series expansion of the analytic function~$f(x)$. Consequently, the number of trees of~$\cT_\bd$ when~$\bd \equiv d$ is given by
$$\frac{m!\theta_d(1+md)}{1+md} = \frac{m!(md)!}{1+md}[y^{md}]\zeta_d(y)^{1+md} = \frac{(md)!(1+md)^{m-1}}{(d!)^m} = \frac{(n-1)!n^{m-1}}{(d!)^m}.$$
In fact, this method based on generating functions can also be extended to obtain the formula of Lemma~\ref{proposition:count} for general~$\bd$, using the multivariate version of Lagrange inversion formula~\cite{BenderRichmond}.

\medskip
Let~$T$ be a tree in~$\cT_\bd$ and~$\varepsilon$ be an edge of~$K_{M,N}$ not in~$T$. There is a unique cycle in~$T \cup \{\varepsilon\}$. Let~$\varepsilon'$ be the other edge in this cycle adjacent to the \svx of~$\varepsilon$. Then~$T$ and~$T' \eqdef (T \setminus \{\varepsilon\}) \cup \{\varepsilon'\}$ are the only two trees in~$\cT_\bd$ contained in~$T \cup \{\varepsilon\}$. We say that~$T'$ is obtained from~$T$ by \defn{pivoting on}~$\varepsilon$. Note that this operation removes the edge~$\varepsilon'$ from~$T$.

We consider the \defn{pivoting graph}~$\cG_{\bd}$ whose vertices are the trees in~$\cT_{\bd}$, and whose edges are pairs of trees which differ by pivoting on an edge. We study the diameter of~$\cG_\bd$.

\begin{theorem}\label{theorem:diameter-bound}
The diameter of~$\cG_\bd$ is at most~$2\sum_\mu \bd(\mu) = 2n-2$.
\end{theorem}

\begin{proof}
Consider two trees $S$ and $T$ in $\cT_\bd$.  We say that a \svx is \defn{consistent} in~$(S,T)$ when all edges incident to it coincide in~$S$ and~$T$. We will prove that we can make all \svx{}s consistent after at most~$2\sum_\mu \bd(\mu)$ pivots in~$S$ and~$T$, which means that the distance between~$S$ and~$T$ in $\cG_\bd$ is at most $2\sum_\mu \bd(\mu)$. We proceed by induction on the number of non-consistent \svx{}s.

Assume that~$S$ and~$T$ are distinct so that there exists at least one non-consistent \svx in~$(S,T)$. We pick any \svx~$\mu_*$ and choose a non-consistent \svx~$\mu$ with maximal distance to~$\mu_*$ in~$S$. Removing~$\mu$ from~$S$ results in~$1+\bd(\mu)$ components. We denote by~$V_0$ the vertex set of the component containing~$\mu_*$ and by~$V_1,\dots, V_{\bd(\mu)}$ the vertex sets of the other components. We also denote by~$\nu_k$ the \dvx\ in~$V_k$ adjacent to~$\mu$ in~$S$. The choice of~$\mu$ ensures that all \svx{}s among~$V_1,\dots, V_{\bd(\mu)}$ are consistent in~$(S,T)$. In particular~$V_1,\ldots,V_{\bd(\mu)}$ induce connected subgraphs of~$T$.

In at most~$2\bd(\mu)$ pivots, we will transform~$S$ into~$\bar S$ and~$T$ into~$\bar T$ such that~$\mu$ becomes consistent in~$(\bar S, \bar T)$ while ensuring that \svx{}s that are consistent in~$(S,T)$ stay consistent in $(\bar S, \bar T)$. 

We will perform our pivots in two phases. For the first phase, let us assume that there is an edge $(\mu,\nu_k)$ of $S$ that is not shared by $T$. The unique cycle in $T\cup\{(\mu,\nu_k)\}$ contains a second edge $(\mu,\nu)$ incident with $\mu$. We distinguish two cases. If $\nu\notin\{\nu_1,\dots,\nu_{\bd(\mu)}\}$ then we perform a pivot on $(\mu,\nu_k)$ in $T$, which removes $(\mu,\nu)$ from~$T$.

\begin{figure}[hbt]
	\centerline{\includegraphics[scale=.7]{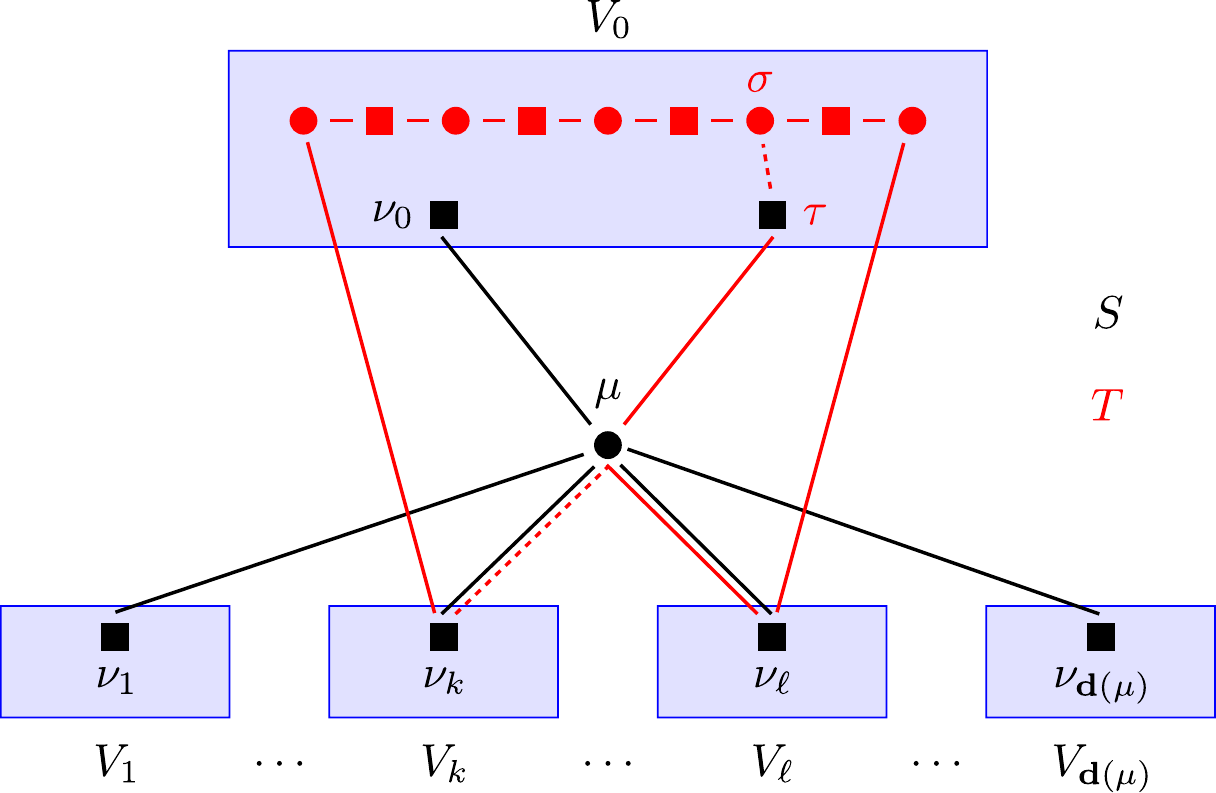}}
	\caption{Phase one pivots on $T$: the cycle in $T \cup (\mu,\nu_k)$ is depicted in red. Pivoting on $(\sigma,\tau)$ in $T$ removes the edge to the right of $\sigma$, producing $T'$. Then pivoting on $(\mu,\nu_k)$ in $T'$ removes $(\mu,\tau)$.}\label{figure:phase1}
\end{figure}

In the other case, when $\nu=\nu_\ell$ for some $\ell\in\{1,\dots,{\bd(\mu)}\}$, we observe that the unique cycle in $T\cup\{(\mu,\nu_k)\}$ contains a non-consistent \svx~$\sigma$ other than~$\mu$; otherwise the cycle would be contained in the tree~$S$. Next, the \svx~$\mu$ has a neighbour $\tau\notin\{\nu_1,\ldots,\nu_{\bd(\mu)}\}$ in $T$. See Figure~\ref{figure:phase1}. We now perform first a pivot on $(\sigma,\tau)$ in $T$. In the tree $T'$ obtained from $T$ by this pivot, we find that inserting the edge $(\mu,\nu_k)$ results in a cycle that contains $(\mu,\tau)$. Consequently, our next pivot, which we perform  on $(\mu,\nu_k)$ in $T'$, removes $(\mu,\tau)$. We continue with this process until all the edges $(\mu,\nu_k)$ lie in $T$. This concludes the first phase. 

After the first phase, either $\mu$ has become consistent or there is exactly one edge left in $T$ that is not shared by $S$. Clearly, the \svx vertex of this edge is~$\mu$, while its \dvx, which we denote by~$\kappa$, is contained in~$V_0$. Then pivoting in~$S$ on~$(\mu, \kappa)$ makes~$\mu$ consistent.

The previous arguments show that the \svx~$\mu$ can be made consistent in at most~$2\bd(\mu) + 1$ pivots: we need at most two pivots to insert~$(\mu, \nu_k)$ in~$T$ for each~${k \ge 1}$, and one final pivot to insert~$(\mu, \kappa)$ in~$S$. To obtain the claimed result, we observe that the current procedure requires~$2\bd(\mu)+1$ pivots only in the case where we start with no edge between~$\mu$ and~$\nu_1,\dots,\nu_{\bd(\mu)}$ in $T$. 
But in this case, we can drop the last~$1$ since at the very beginning, we only need one pivot to insert~$(\mu, \nu_1)$ in $S$.
\end{proof}


\section{Background on transportation polytopes}\label{section:background}

Given a \defn{supply function}~$\bm : M \to \bbR_{\ge 0}$ on a set~$M$ of~$m$ \defn{\svx{}s} and a \defn{demand function}~$\bn : N \to \bbR_{\ge 0}$ on a set~$N$ of~$n$ \defn{\dvx{}s}, the \defn{transportation polytope}~$P_{\bm,\bn}$ has the following inequality description:
\begin{equation}\label{eqn:inequalityDescription}
\forall \mu, \; \forall \nu, \quad x_{\mu,\nu} \ge 0,\quad \sum\nolimits_{\nu'} x_{\mu,\nu'} = \bm(\mu), \quad \text{and} \quad \sum\nolimits_{\mu'} x_{\mu',\nu} = \bn(\nu).
\end{equation}
Again, we assume that~$P_{\bm,\bn}$ is non-empty, \ie that~${\sum_\mu \bm(\mu) = \sum_\nu \bn(\nu)}$. 
Since $P_{\bm,\bn}$ is defined by~$m+n-1$ linearly independent equations in an~$mn$-dimensional ambient space, it has dimension~$\delta \eqdef (m-1)(n-1)$.

From the inequality description~(\ref{eqn:inequalityDescription}), the $m \times n$ transportation polytope~$P_{\bm,\bn}$ has at most~$mn$ facets, each defined by an inequality of the form~$x_{\mu,\nu} \ge 0$. Reciprocally, the following lemma affirms that an inequality $x_{\mu,\nu}$ defines a facet as soon as it is tight.
\begin{lemma}\label{lemma:facet-characterization}
Let $mn > 4$. The inequality~$x_{\mu,\nu} \ge 0$ defines a facet of~$P_{\bm,\bn}$ if and only if there is an~$x \in P_{\bm,\bn}$ such that~$x_{\mu,\nu} = 0$.
\end{lemma}
This follows from \eg Theorem~3.1 in Chapter~6 of~\cite{YemelichevKovalevKravtsov}, which states:
\begin{proposition}[\cite{YemelichevKovalevKravtsov}, pg.~271]
Let $mn >4$. The inequality~$x_{\mu,\nu} \ge 0$ defines a facet of~$P_{\bm,\bn}$ if and only if $\bm(\mu) + \bn(\nu) < \sum_{\mu' \in M} \bm(\mu')$
\end{proposition}

Let~$x \in P_{\bm,\bn}$. The \defn{support} of~$x$ is the subgraph~$\supp(x)$ of~$K_{M,N}$ consisting of the edges~$(\mu,\nu)$ for which~$x_{\mu,\nu} > 0$. The point~$x$ lies in~$mn-|\supp(x)|$ many facets of~$P_{\bm,\bn}$. In this article, we assume that~$P_{\bm,\bn}$ is simple, \ie that every vertex has precisely~$m+n-1$ non-zero coordinates, or equivalently, that no proper subsets~$M' \subsetneq M$ and~$N' \subsetneq N$ satisfy~$\sum_{\mu \in M'} \bm(\mu) = \sum_{\nu \in N'} \bn(\nu)$. 
The following proposition summarizes the properties of the supports of the vertices of~$P_{\bm,\bn}$ (for a proof, see \eg~\cite{YemelichevKovalevKravtsov}).

\begin{proposition}
Let~$P_{\bm,\bn}$ be a simple transportation polytope.
\begin{enumerate}[(i)]
\item A point~$x \in P_{\bm,\bn}$ is a vertex of the polytope~$P_{\bm,\bn}$ if and only if~$\supp(x)$ is a spanning tree of~$K_{M,N}$.
\item A vertex~$x$ of~$P_{\bm,\bn}$ is determined by its support~$\supp(x)$.
\item The supports of two adjacent vertices of~$P_{\bm,\bn}$ differ in precisely two edges.
\end{enumerate}
\end{proposition}

We close this background section with a relevant family of transportation polytopes. The \defn{Birkhoff polytope}~$B_m$ of size~$m$ is the transportation polytope whose supply and demand functions are both constant to~$m$. Its vertices are precisely the permutation matrices. Note that the support of a permutation matrix is a perfect matching, so that the classical Birkhoff polytope is not simple. We also want to underline that the Birkhoff polytope~$B_m$ is known to have diameter exactly~$2$ (see Theorem 1.7 of Chapter 5 in~\cite{YemelichevKovalevKravtsov}). The $m \times n$ \defn{generalized Birkhoff polytope}~$B_{m,n}$ is the transportation polytope whose supply function is constant to~$n$ and whose demand function is constant to~$m$. Note that it is a simple polytope if and only if $m$~and~$n$ are relatively prime. As a consequence of the next section, we will obtain the diameter of generalized Birkhoff polytopes~$B_{m,1+md}$ for $m,d \ge 1$.


\section{A special class of transportation polytopes}\label{section:transportationPolytopes}

In this section, we study a specific family of transportation polytopes whose combinatorial properties are closely related to that of the family~$\cT_\bd$ of trees defined in Section~\ref{section:trees}.
We continue to use the notation of Section~\ref{section:trees}.

Let~$\bm : M \to \bbR_{\ge 0}$ be the \defn{supply function} defined by~$\bm(\mu) = 1+ m\bd(\mu)$ for all~$\mu \in M$, and let~$\bn : N \to \bbR_{\ge 0}$ be the \defn{demand function} defined by~$\bn(\nu) = m$ for all~$\nu \in N$. Then the~$m \times n$ transportation polytope~$P_\bd \eqdef P_{\bm,\bn}$ is non-empty (since~$\sum_\mu \bm(\mu) = \sum_\nu \bn(\nu)$) and simple (since~${\sum_{\mu \in M'} \bm(\mu) \ne \sum_{\nu \in N'} \bn(\nu)}$ for any proper subsets~$M' \subsetneq M$ and~${N' \subsetneq N}$).

\begin{example}
If~$\bd \equiv d$, then~$n = 1 + md$ and we have~$\bm \equiv n$ and~$\bn \equiv m$. The polytope~$P_\bd$ is thus the generalized Birkhoff polytope~$B_{m,n}$ of size~$m \times n$. Note that $m$~and~$n = 1 + md$ are relatively prime, so that~$P_\bd = B_{m,n}$ is simple.
\end{example}

The following lemma relates the family~$\cT_\bd$ with the transportation polytope~$P_\bd$.

\begin{lemma}\label{lemma:bijection}
The support function~$x \mapsto \supp(x)$ is a bijection from the vertex set of~$P_\bd$ to the set~$\cT_\bd$ of trees, which sends the~$1$-skeleton of~$P_\bd$ to the graph~$\cG_\bd$.
\end{lemma}

\begin{proof}
Let~$T = \supp(x)$ for a vertex~$x$ of~$P_\bd$. The degree of a \svx~$\mu$ in~$T$ is at least~$\bd(\mu)$ for any \svx~$\mu$.  Since~$T$ has precisely~$m + n - 1 = \sum_\mu \bd(\mu)$ edges, each \svx~$\mu$ has degree exactly~$\bd(\mu)$ in~$T$, and thus~$T$ is in~$\cT_\bd$. 

We prove the reciprocal statement by induction on~$m+n$. Let~$T \in \cT_\bd$. We need to exhibit a point~$x \in P_{\bm,\bn}$ whose support is~$T$. Note that~$x$ will automatically be a vertex of~$P_{\bm,\bn}$ since its support is a tree (see \eg~\cite{YemelichevKovalevKravtsov}).

Assume first that a \dvx~$\tau$ is a leaf in~$T$ only adjacent to a \svx~$\sigma$. Let~${N' = N \setminus \{\tau\}}$, let~$T'$ be the restriction of~$T$ to~$M \cup N'$, and define~$\bd' : M \to \bbN_{\ge 0}$ by~$\bd'(\sigma) = \bd(\sigma)-1$ and~$\bd'(\mu) = \bd(\mu)$ if~$\mu \ne\sigma$. Since~$T' \in \cT_{\bd'}$, the induction hypothesis ensures the existence of a vertex~$x' \in P_{\bd'} \subset \bbR^{M \times N'}$ whose support is~$T'$. We define a point~$x \in \bbR^{M \times N}$ by extending~$x'$ with~$x_{\sigma,\tau} = \bn(\tau) = m$ and~$x_{\mu,\tau} = 0$ for~$\mu \ne \sigma$. Then~$x$ is a vertex of~$P_\bd$ with support~$T$.

Assume now that a \svx~$\sigma$ is a leaf in~$T$ only adjacent to a \dvx~$\tau$. Let~${M' = M \setminus \{\sigma\}}$, let~$T'$ be the restriction of~$T$ to~$M' \cup N$, and let~$\bd'$ be the restriction of~$\bd$ to~$M'$. Since~$T' \in \cT_{\bd'}$, the induction hypothesis ensures the existence of a vertex~$x' \in P_{\bd'} \subset \bbR^{M' \times N}$ whose support is~$T'$. 
We orient the edges of~$T'$ towards~$\tau$. We define a point~$x \in \bbR^{M \times N}$ by: 
\begin{itemize}
\item
$x_{\sigma,\tau} = 1$, 
\item 
$x_{\mu,\nu} = x'_{\mu,\nu}$ if there is an arc oriented from~$\mu$ to~$\nu$ in~$T'$, 
\item 
$x_{\mu,\nu} = x'_{\mu,\nu}+1$ if there is an arc oriented from~$\nu$ to~$\mu$ in~$T'$, 
\item
$x_{\mu,\nu} = 0$ if~$(\mu,\nu)$ is not an edge of~$T$.
\end{itemize}
Then~$x$ is a vertex of~$P_\bd$ with support~$T$.

We thus proved that the support function~$x \mapsto \supp(x)$ is a bijection from the vertices of~$P_\bd$ to~$\cT_\bd$. For the second part of the statement, observe that two vertices of~$P_\bd$ are adjacent if and only if their supports differ in precisely two edges, \ie in a pivot.
\end{proof}
More generally, the support function provides an isomorphism from the face lattice of the polar polytope of~$P_\bd$ to the simplicial complex whose maximal faces are the complements of the trees in~$\cT_\bd$.

\begin{lemma}
If at least two \svx{}s have prescribed degree at least~$2$ (in other words, if~$|\bd^{-1}(\bbN_{\ge 1})| \ge 2$), then the number~$\phi$ of facets of~$P_\bd$ is precisely~$mn$.
\end{lemma}

\begin{proof}
By Lemma~\ref{lemma:facet-characterization}, the inequality~$x_{\mu,\nu} \ge 0$ defines a facet of~$P_\bd$ if and only if there exists a vertex~$x \in P_\bd$ with~$x_{\mu,\nu} = 0$. Such a vertex exists because there is a tree~$T \in \cT_\bd$ with~$(\mu,\nu) \notin T$ since~$1 + \bd(\mu) < n$.
\end{proof}
In the conditions of this lemma, the Hirsch bound for the diameter is thus ${H \eqdef \phi - \delta = m+n-1}$. Applying the result of Theorem~\ref{theorem:diameter-bound}, we obtain:

\begin{theorem}
The diameter of~$P_\bd$ is at most~$2n - 2 \le 2H$. In particular, if~$\bd \equiv 1$, then the diameter of~$P_\bd = B_{m,1+m}$ is at most the Hirsch bound~$H$.
\end{theorem}

Note that the generalized Birkhoff polytope~$B_{2,3}$ is an hexagon. Consequently, its diameter is strictly less than the Hirsch bound~$H$. This situation does not occur for larger $m \times (1+m)$ generalized Birkhoff polytopes:

\begin{corollary}
For any~$m \ge 3$, the diameter of the $m \times (1+m)$ generalized Birkhoff polytope~$B_{m,1+m}$ coincides with the Hirsch bound~$2m$.
\end{corollary}
\begin{proof}
The upper bound is given by the previous corollary. For the lower bound, it is easy to construct two trees of~$\cT_\bd$ with disjoint edge sets.
\end{proof}

Pak~\cite{PakIgor} investigates the $f$-vector for $m \times n$ generalized Birkhoff polytopes when $n = m+1$. Using Proposition~\ref{proposition:count} and the simplicity of~$P_\bd$, we obtain the first two values of the $f$-vector for a slightly more general class of transportation polytopes:
\begin{corollary}
The number of vertices and edges of~$P_\bd$ are respectively
\[ 
f_0(P_\bd) = \frac{(n-1)! n^{m-1}}{\prod_\mu \bd(\mu)!} \quad \text{and} \quad f_1(P_\bd) = \frac{(m-1)(n-1)(n-1)! n^{m-1}}{2\prod_\mu \bd(\mu)!}.
\]
\end{corollary}

\bibliographystyle{alpha}
\bibliography{diameterGeneralizedBirkhoffPolytope}

\end{document}